\documentclass[a4paper,reqno,9pt,oneside]{amsart}
\usepackage{amsmath,geometry}
\usepackage{color}

\numberwithin{equation}{section}

\newtheorem{thm}{Theorem}[section]

\newtheorem{pro}[thm]{Proposition}

\newtheorem{oss}[thm]{Remark}
\newtheorem{cor}[thm]{Corollary}

\begin{document}

\date{\today}

\title[Isolated singularities for the $n-$Liouville equation]{Isolated singularities for the $n-$Liouville equation}

\author{Pierpaolo Esposito}
\address{Pierpaolo Esposito, Dipartimento di Matematica e Fisica, Universit\`a degli Studi Roma Tre',\,Largo S. Leonardo Murialdo 1, 00146 Roma, Italy}
\email{esposito@mat.uniroma3.it}

\thanks{Partially supported by Gruppo Nazionale per l'Analisi Matematica, la Probabilit\'a  e le loro Applicazioni (GNAMPA) of the Istituto Nazionale di Alta Matematica (INdAM)}

\begin{abstract} In dimension $n$ isolated singularities -- at a finite point or at infinity-- for solutions of finite total mass to the $n-$Liouville equation are of logarithmic type. As a consequence, we simplify the classification argument in \cite{Esp} and establish a quantization result for entire solutions of the singular $n-$Liouville equation.
\end{abstract}

\maketitle

\section{Introduction}
\noindent The behavior near an isolated singularity has been discussed by Serrin in \cite{Ser1,Ser2} for a very general class of second-order quasi-linear equations. The simplest example is given by the prototypical equation $-\Delta_n u=f$, where $\Delta_n(\cdot) =\hbox{div} (|\nabla (\cdot) |^{n-2}\nabla (\cdot) )$, $n \geq 2$, is the $n-$Laplace operator. In dimension $n$, the case $f \in L^1$ is very delicate as it represents a limiting situation where Serrin's results do not apply. We will be interested in the $n-$Liouville equation, where $f$ is taken as an exponential function of $u$ according to Liouville's seminal paper \cite{Lio}, and the singularity might be at a finite point or at infinity.

\medskip \noindent To this aim, it is enough to consider the generalized $n-$Liouville equation
\begin{equation}\label{E1}
-\Delta_n u=|x|^{n \alpha} e^u \mbox{ in } \Omega \setminus \{0\}, \: \int_{\Omega} |x|^{n \alpha} e^u <+\infty
\end{equation}
on an open set $\Omega \subset \mathbb{R}^n$ with $0 \in \Omega$, and we will be concerned with describing the behavior of $u$ at $0$. A solution $u$ of \eqref{E1} stands for a function $u \in C^{1,\eta}_{loc}(\Omega \setminus \{0\})$ which satisfies
$$\int_{\Omega} |\nabla u|^{n-2}\langle \nabla u, \nabla \varphi\rangle=\int_{\Omega} |x|^{n \alpha} e^u \varphi
\qquad \forall \ \varphi \in  C_0^1(\Omega \setminus \{0\}). $$
The regularity assumption on $u$ is not restrictive since a solution in $W^{1,n}_{\hbox{loc}}(\Omega \setminus \{0\})$ is automatically in $C^{1,\eta}_{loc}(\Omega \setminus \{0\})$, for some $\eta \in (0,1)$, thanks to \cite{Dib,Ser1,Tol}, see Theorem 2.3 in \cite{Esp}.

\medskip \noindent Concerning the behavior near an isolated singularity, our main result is
\begin{thm} \label{thm1} Let $u$ be a solution of \eqref{E1}. Then there exists $ \gamma >- n^n |\alpha+1|^{n-2}(\alpha+1) \omega_n $, $\omega_n=|B_1(0)|$, so that
\begin{equation}\label{E1bis}
-\Delta_n u=|x|^{n \alpha} e^u -\gamma \delta_0 \hbox{ in } \Omega 
\end{equation}
with
\begin{equation} \label{sing1}
u-\gamma  (n\omega_n |\gamma|^{n-2})^{-\frac{1}{n-1}} \log |x|\in L^\infty_{\hbox{loc}} (\Omega)
\end{equation}
and
\begin{equation} \label{gamma1}
\lim_{x \to 0} \left[|x|\nabla u(x)-\gamma  (n\omega_n |\gamma|^{n-2})^{-\frac{1}{n-1}} \frac{x}{|x|} \right] =0.
\end{equation}
\end{thm}

\medskip \noindent The case $\alpha=-2$ is relevant for the asymptotic behavior at infinity for solutions $u$ of
\begin{equation}\label{E1entire}
-\Delta_n u=e^u \mbox{ in } \Omega  , \: \int_\Omega e^u<+\infty,
\end{equation}
where $\Omega$ is an unbounded open set so that $B_R(0)^c \subset \Omega$ for some $R>0$. Indeed, let us recall that $\Delta_n$ is invariant under Kelvin transform: if $u$ solves \eqref{E1}, then $\hat u(x)=u(\frac{x}{|x|^2})$ does satisfy
\begin{equation} \label{Kelvin}
-\Delta_n \hat u=|x|^{-2n} (-\Delta_n u) (\frac{x}{|x|^2})=|x|^{-n(\alpha+2)} e^{\hat u} \hbox{ in }\hat \Omega=\{x \not=0 :\ \frac{x}{|x|^2} \in \Omega\}.
\end{equation}
By Theorem \ref{thm1} applied with $\alpha=-2$ to $\hat u$ at $0$ we find:
\begin{cor} \label{cor1} Let $u$ be a solution of \eqref{E1entire} on an unbounded open set $\Omega$ with $B_R(0)^c \subset \Omega$ for some $R>0$. Then there holds
\begin{equation} \label{1455}
u=-\left(\frac{\gamma_\infty}{n\omega_n} \right)^{\frac{1}{n-1}} \log |x|+O(1)
\end{equation}
as $|x| \to \infty$ for some $ \gamma_\infty >n^n \omega_n $. In particular, when $\Omega=\mathbb{R}^n$ there holds
\begin{equation} \label{145}
u= - \left(\frac{1}{n \omega_n}\int_{\mathbb{R}^n} e^u \right)^{\frac{1}{n-1}}\log |x|+O(1)
\end{equation}
as $|x| \to \infty$.
\end{cor}
\noindent When $n=2$ the asymptotic expansion \eqref{145} is a well known property established in \cite{ChLi} by means of the Green representation formula -- unfortunately not available in the quasi-linear case-- and of the growth properties of entire harmonic functions. Notice that
$$\gamma_\infty=\int_{\mathbb{R}^n} |x|^{-2n} e^{\hat u} =\int_{\mathbb{R}^n} e^u$$
follows by integrating \eqref{E1bis} written for $\hat u$ on $\mathbb{R}^n$. Property \eqref{145} has been already proved in \cite{Esp} under the assumption $\gamma_\infty>n^n \omega_n$ and the present full generality allows to simplify the classification argument in \cite{Esp}: a Pohozaev identity leads for $\gamma_\infty$ to the quantization property 
\begin{equation}\label{E16ter}
\int_{\mathbb{R}^n} e^u=n (\frac{n^{2}}{n-1})^{n-1} \omega_n
\end{equation}
and an isoperimetric argument concludes the classification result thanks to \eqref{E16ter}. 

\medskip \noindent In the punctured plane $\Omega =\mathbb{R}^n \setminus \{0\}$ the isoperimetric argument fails and in general the classification result is no longer available. The two-dimensional case $n=2$ has been treated via complex analysis in \cite{ChWa,PrTa}: solutions $u$ to
\begin{equation}\label{E1bisentiren=2}
-\Delta u=e^u -\gamma \delta_0\mbox{ in } \mathbb{R}^2,\: \int_{\mathbb{R}^2} e^u<+\infty,
\end{equation}
have been classified for $\gamma>-4\pi$ of the form
$$u(x)=\log\frac{8 (\alpha+1)^2 \lambda^2 |x|^{2\alpha }}{(1+\lambda^2 |x^{\alpha+1}+c|^2)^2},\: \alpha=\frac{\gamma}{4\pi},$$
with $\lambda>0$ and a complex number $c =0$ if $\alpha \notin \mathbb{N}$, and in particular satisfy
\begin{equation} \label{1059}
\int_{\mathbb{R}^n} e^u=8\pi(\alpha+1).
\end{equation}
The structure of entire solutions $u$ to \eqref{E1bisentiren=2} changes drastically passing from radial solutions when $\alpha \notin \mathbb{N}$ to non-radial solutions when $\alpha \in \mathbb{N}$ (and $c \not=0$). Unfortunately a PDE approach is not available for $n=2$ and a classification result is completely out of reach when $n \geq 3$. However, quantization properties are still in order as it follows by Theorem \ref{thm1} and the Pohozaev identities:

\begin{thm} \label{thm2} Let $u$ be a solution of 
\begin{equation}\label{E1bisentire}
-\Delta_n u=e^u -\gamma \delta_0\mbox{ in } \mathbb{R}^n,\: \int_{\mathbb{R}^n} e^u<+\infty.
\end{equation}
Then $\gamma>-n^n \omega_n$ and
\begin{equation}\label{quant}
\int_{\mathbb{R}^n} e^u=\gamma +\gamma_\infty
\end{equation}
with $\gamma_\infty$ the unique solution in $(n^n \omega_n,+\infty)$ of
\begin{equation} \label{921}
\frac{n-1}{n} (n\omega_n)^{-\frac{1}{n-1}}   \gamma_\infty^{\frac{n}{n-1}} -n  \gamma_\infty=n\gamma+
\frac{n-1}{n} (n\omega_n)^{-\frac{1}{n-1}}  |\gamma|^{\frac{n}{n-1}} .
\end{equation}
\end{thm}
\noindent When $n=2$ notice that for $\gamma> -4\pi$ the unique solution $\gamma_\infty>4\pi$ of \eqref{921} is given explicitly as $\gamma_\infty=\gamma+8\pi$ and then $\int_{\mathbb{R}^n} e^u=2\gamma +8\pi=8\pi(\alpha+1)$ in accordance with \eqref{1059}. To have Theorem \ref{thm2} meaningful, in Section 4 we will show the existence of a family of radial solutions $u$ to \eqref{E1bisentire} but we don't know whether other solutions might exist or not, depending on the value of $\gamma$. Notice that \eqref{E1bisentiren=2} is equivalent to
$$-\Delta v=|x|^{2\alpha} e^v \mbox{ in } \mathbb{R}^2,\: \int_{\mathbb{R}^2} |x|^{2\alpha} e^v<+\infty,$$
in terms of $v=u-2\alpha \log |x|$. For $n\geq 3$ such equivalence breaks down and the problem
\begin{equation}\label{E1sim}
-\Delta_n v=|x|^{n \alpha} e^v  \mbox{ in } \mathbb{R}^n,\: \int_{\mathbb{R}^n} |x|^{n \alpha} e^v<+\infty,
\end{equation}
has its own interest, independently of \eqref{E1bisentire}. As in Theorem \ref{thm2}, for \eqref{E1sim} we have the following quantization result:
\begin{thm} \label{thm3} Let $v$ be a solution of  \eqref{E1sim}. Then $\alpha > -1$ and
$$\int_{\mathbb{R}^n} |x|^{n \alpha}e^v=n (\frac{n^{2}}{n-1})^{n-1} (\alpha+1)^{n-1} \omega_n.$$
\end{thm}
\noindent Radial solutions $v$ of \eqref{E1sim} can be easily obtained as $v=n \log (\alpha+1)+u (|x|^{\alpha+1})$ in terms of a radial entire solution $u$ to \eqref{E1entire}. Thanks to the classification result in \cite{Esp}, for \eqref{E1sim} we can therefore exhibit the following family of radial solutions:
$$v_\lambda=\log\frac{c_n (\alpha+1)^n \lambda^n}{(1+\lambda^{\frac{n}{n-1}} |x|^{\frac{n(\alpha+1)}{n-1}})^n}, \: c_n=n (\frac{n^{2}}{n-1})^{n-1}.$$

\medskip \noindent Problems with exponential nonlinearities on a bounded domain with given singular sources can exhibit non-compact solution-sequences, whose shape near a blow-up point is asymptotically described by the limiting problem \eqref{E1bisentire}. In the regular case (i.e. in absence of singular sources) a concentration-compactness principle has been established \cite{BrMe} for $n=2$  and \cite{AgPe} for $n \geq 2$. In the non-compact situation the exponential nonlinearity concentrates at the blow-up points as a sum of Dirac measures. Theorem \ref{thm2} gives information on the concentration mass of such Dirac measures at a singular blow-up point, which is expected bo te a super-position of several masses $c_n\omega_n$ carried by multiple sharp collapsing peaks governed by \eqref{E1bisentire}$_{\gamma=0}$ and possibly the mass \eqref{quant} of a sharp peak described by \eqref{E1bisentire}. In the regular case such quantization property on the concentration masses has been proved \cite{LiSh} for $n=2$ and extended \cite{EsMo} to $n \geq 2$ by requiring an additional boundary assumption, while the singular case has been addressed in \cite{BaTa,Tar} for $n=2$. For Theorem \ref{thm3} a similar comment is in order.

\medskip \noindent Let us briefly explain the main ideas behind Theorem \ref{thm1}. We can re-adapt the argument in \cite{Esp} to show that $u \in \displaystyle \bigcap_{1\leq q <n} W^{1,q}_{\hbox{loc}}(\Omega)$ and then $u$ satisfies \eqref{E1bis} for some $\gamma \in \mathbb{R}$. On a radial ball $B \subset \subset \Omega$ decompose $u$ as $u=u_0+h$, where $h$ is given by
$$\Delta_n h= \gamma \delta_0 \hbox{ in } B,\quad h=u \hbox{ on }\partial B$$
and satisfies  \eqref{sing1} thanks to \cite{KiVe,Ser1,Ser2}. The key property stems from a simple observation: 
$|x|^{n\alpha}e^h \in L^1$ near $0$ implies $|x|^{n\alpha}e^h \in L^p$ near $0$ for some $p>1$ whenever $h$ has a logarithmic singularity at $0$. Back to \cite{Ser1,Ser2} thanks to such improved integrability, one aims to show that $u_0 \in L^\infty(B)$ and then $u$ has the same logarithmic behavior \eqref{sing1} as $h$. In order to develop a regularity theory for the solution $u_0$ of
\begin{equation}\label{929}
-\Delta_n (u_0+h)+\Delta_n h=|x|^{n \alpha} e^{u_0+h}  \mbox{ in } B, \: u_0=0 \hbox{ on } \partial B,
\end{equation}
the crucial point is to establish several integral inequalities involving $u_0$ paralleling the estimates available for entropy solutions in \cite{AgPe,BBGGPV} and for $W^{1,n}-$solutions in \cite{Ser1}. To this aim, we make use of the deep uniqueness result \cite{DHM} to show that $u$ can be regarded as a Solution Obtained as Limit of Approximations (the so-called SOLA, see for example \cite{BoGa}).

\medskip \noindent The paper is organized as follows. In Section 2 we develop the above argument to prove Theorem \ref{thm1}. Section 3 is devoted to establish Theorems \ref{thm2}-\ref{thm3} via Pohozaev identities: going back to an idea of Y.Y. Li and N. Wolanski for $n=2$, the Pohozaev identities have revealed to be a fundamental tool to derive information on the mass of a singularity (see for example \cite{BaTa,EsMo,RoWe}). In Section 4 a family of radial solutions $u$ to \eqref{E1bisentire} is constructed.

\vskip.6cm\noindent
{\bf Acknowledgements:} The author would like to thank the referee for a careful reading and for pointing out a mistake in the original argument.

\section{Proof of Theorem \ref{thm1}}
\noindent Assume $B_1(0) \subset \subset \Omega$.
Let us first establish the following property on $u$.
\begin{pro} Let $u$ be a solution of \eqref{E1}. There exists $C>0$ so that
\begin{equation} \label{1340}
u(x) \leq C- n(\alpha+1) \log |x| \qquad \hbox{in } B_1(0) \setminus \{ 0 \}.
\end{equation}
\end{pro} 
\begin{proof} Letting $U_r(y)=\hat u (\frac{y}{r})+n(\alpha+1) \log r=u(\frac{ry}{|y|^2})+n(\alpha+1) \log r$ for $0<r\leq \frac{1}{2} $, we have that $U_r$ solves
\begin{equation} \label{1156}
-\Delta_n U_r=|y|^{-n(\alpha+2)}  e^{U_r} \qquad\hbox{in } \mathbb{R}^n \setminus B_{\frac{1}{2}}(0), \quad
\int_{\mathbb{R}^n \setminus B_{\frac{1}{2}}(0)} |y|^{-n(\alpha+2)}  e^{U_r}= \int_{B_{2r}(0)} |x|^{n\alpha}  e^u
\end{equation}
in view of \eqref{Kelvin}. Given a ball $B_{\frac{1}{2}}(x_0)$ for  $x_0 \in \mathbb{S}^{n-1}$, let us consider the $n-$harmonic function $H_r$ in $B_{\frac{1}{2}}(x_0)$ so that $H_r=U_r$ on $\partial B_{\frac{1}{2}}(x_0)$. By the weak maximum principle we deduce that $H_r \leq U_r$ in $B_{\frac{1}{2}}(x_0)$ and then
\begin{equation} \label{1157}
\int_{B_{\frac{1}{2}}(x_0)} (H_r)_+^n \leq  \int_{B_{\frac{1}{2}}(x_0)} (U_r)_+^n \leq n! \int_{B_{\frac{1}{2}}(x_0)} e^{U_r} \leq C \int_{\mathbb{R}^n \setminus B_{\frac{1}{2}}(0)} |y|^{-n(\alpha+2)}  e^{U_r} \leq C \int_{\Omega} |x|^{n\alpha}  e^u<+\infty
\end{equation}
for all $0<r\leq \frac{1}{2}$. By the estimates in \cite{Ser1} we have that there exists $C>0$ so that 
\begin{equation} \label{1316}
\|H_r\|_{\infty, B_{\frac{1}{4}} (x_0)} \leq C
\end{equation}
for all $0<r\leq \frac{1}{2}$. At the same time, by the exponential estimate in \cite{AgPe} we have that there exist $0<r_0 \leq \frac{1}{2}$ and $C>0$ so that
\begin{equation} \label{1330}
\int_{B_{\frac{1}{2}}(x_0)}  e^{2 |U_r-H_r|} \leq C
\end{equation} 
for all $0<r \leq r_0$ in view of $\displaystyle \lim_{ r \to 0^+} \int_{B_{\frac{1}{2}}(x_0)} |y|^{-n(\alpha+2)}  e^{U_r}= 0$ thanks to \eqref{E1} and \eqref{1156}. Since $|y|^{-n(\alpha+2)}  e^{U_r} \leq   C e^{|U_r-H_r|}$ on  $ B_{\frac{1}{4}} (x_0)$ for all $0<r\leq \frac{1}{2}$ in view of \eqref{1316}, we deduce that $|y|^{-n(\alpha+2)}  e^{U_r}$ and $(U_r)_+^{\frac{n}{2}}$ are uniformly bounded in $L^2( B_{\frac{1}{4}} (x_0))$ for all $0<r \leq r_0$ in view of \eqref{1157} and \eqref{1330}. Thanks again to the estimates in  \cite{Ser1}, we finally deduce that
\begin{equation} \label{1333}
\|U_r^+\|_{\infty, B_{\frac{1}{8}} (x_0)} \leq C
\end{equation}
for all $0<r \leq r_0$. Since $\mathbb{S}^{n-1}$ can be covered by a finite number of balls $B_{\frac{1}{8}} (x_0)$, $x_0 \in \mathbb{S}^n$,  going back to $u$ from \eqref{1333} one deduces that
$$u(x)  \leq C-n(\alpha+1) \log |x|$$
for all $|x|=r \leq r_0$. Since this estimate does hold in $B_1(0) \setminus B_{r_0}(0)$ too, we have established the validity of \eqref{1340}.
\end{proof}

\medskip \noindent From now on, set $B=B_r(0)$ for $0<r\leq 1$. We are now ready to establish the starting point for the argument we will develop in the sequel. There holds
\begin{pro} Let $u$ be a solution of \eqref{E1}. Then
\begin{equation} \label{842}
u \in \bigcap_{1\leq q<n} W^{1,q}(\Omega).
\end{equation}
\end{pro} 
\begin{proof} Let us go through the argument in \cite{Esp} to obtain $W^{1,q}-$estimates on $u$. For $0<\epsilon<r<1$ let us introduce $h_{\epsilon,r} \in W^{1,n}(A_{\epsilon,r}) $, $A_{\epsilon,r}:=B \setminus \overline{B_\epsilon(0)}$, as the solution of
$$\Delta_n h_{\epsilon,r}=0 \hbox{ in }A_{\epsilon,r} ,\: h_{\epsilon,r}= u \hbox{ on }\partial A_{\epsilon,r}.$$ Regularity issues for quasi-linear PDEs involving $\Delta_n$ are well established since the works of DiBenedetto, Evans, Lewis, Serrin, Tolksdorf, Uhlenbeck, Uraltseva. For example, by \cite{Dib,Lie,Ser1,Tol} we deduce that $h_{\epsilon,r},\: u_{\epsilon,r}=u-h_{\epsilon,r} \in C^{1,\eta}(\overline{A_{\epsilon,r}})$ and $u_{\epsilon,r}$ satisfies
\begin{equation} \label{907}
-\Delta_n (u_{\epsilon,r}+h_{\epsilon,r})+\Delta_n h_{\epsilon,r}= |x|^{n\alpha}e^ u \hbox{ in }A_{\epsilon,r},\: u_{\epsilon,r}=0 \hbox{ on }\partial A_{\epsilon,r}. 
\end{equation}
By the techniques in \cite{AgPe,BBGGPV,BoGa} we have the following estimates, see Proposition 2.1 in \cite{Esp}:  for all $1\leq q<n$ and all $p\geq 1$ there exist $0<r_0<1$ and $C>0$ so that
\begin{equation} \label{17255}
\int_{A_{\epsilon,r}} |\nabla u_{\epsilon,r}|^q +\int_{A_{\epsilon,r}} e^{p u_{\epsilon,r}} \leq C
\end{equation} 
for all $0<\epsilon<r \leq r_0$ thanks to \eqref{907} and $\displaystyle \lim_{r \to 0^+} \int_{B}  |x|^{n \alpha} e^u =0.$
Since by the Sobolev embedding $W^{1,\frac{n}{2}}_0(B_1(0)) \hookrightarrow L^n(B_1(0))$ there holds $\int_{A_{\epsilon,r}} |u_{\epsilon,r}|^n \leq C$ for all $0<\epsilon<r \leq r_0$  in view of \eqref{17255} with $q=\frac{n}{2}$ and $A_{\epsilon,r} \subset B_1(0)$, we have that
$$\|h_{\epsilon,r} \|_{L^n(A)} \leq C(A) \qquad  \forall \ A \subset \subset \overline{B} \setminus \{0\}\, , \: \forall \ 0<\epsilon<r \leq r_0$$
in view of $u \in C^{1,\eta}_{loc}(B_1(0) \setminus \{0\})$ and then
$$\|h_{\epsilon,r}\|_{C^{1,\eta}(A)} \leq C(A) \qquad \forall \ A \subset \subset \overline{B} \setminus \{0\}\, , \: \forall \ 0<\epsilon<r \leq r_0$$
thanks to \cite{Dib,Lie,Ser1,Tol}. By the Ascoli-Arzel\'a's Theorem and a diagonal process we can find a sequence $\epsilon \to 0$ so that $h_{\epsilon,r} \to h_r$ and $u_{\epsilon,r} \to u_r:=u-h_r$ in $C^1_{\hbox{loc}} (\overline{B}\setminus \{0\})$ as $\epsilon \to 0$, where $h_r \leq u$ is a $n-$harmonic function in $B \setminus \{0\}$ and $u_r$ satisfies
\begin{equation} \label{pro1}
u_r \in W^{1,q}_0(B) ,\qquad e^{u_r} \in L^p(B)
\end{equation}
for all $1\leq q<n$ and all $p\geq 1$ if $r$ is sufficiently small in view of \eqref{17255}. Since 
$$h_r(x) \leq C- n(\alpha+1) \log |x| \qquad \hbox{in }B$$
in view of $h_r \leq u$ and  \eqref{1340}, we have that $H^\lambda(y)=- \frac{h_r(\lambda y)}{\log \lambda}$ is a $n-$harmonic function in $B_{\frac{r}{\lambda}}(0)$ so that $H^\lambda \leq  n(\alpha+1)+1$ in $B_2(0)\setminus B_{\frac{1}{2}}(0)$ for all $0<\lambda\leq \lambda_0$, where $\lambda_0 \in (0,\frac{r}{2}]$ is a suitable small number. By the Harnack inequality in Theorem 7-\cite{Ser1} applied to $n(\alpha+1)+1-H^\lambda \geq 0$ we deduce that 
\begin{equation} \label{Harn}
\max_{|y|=1} H^\lambda \leq C\left[\frac{n|\alpha+1|+1}{C}+\min_{|y|=1} H^\lambda \right]
\end{equation}
for all $0<\lambda\leq \lambda_0$, for a suitable $C\in (0,1)$. There are two possibilities:
\begin{itemize}
\item either $ \displaystyle \min_{|y|=1} H^\lambda \geq -\frac{n|\alpha+1|+1}{C}$ for all $0<\lambda\leq \lambda_1$ and some $\lambda_1\in(0, \lambda_0]$, which implies $\displaystyle \max_{|y|=1} |H^\lambda| \leq \frac{n|\alpha+1|+1}{C}$ for all $0<\lambda\leq \lambda_1$ and in particular
\begin{equation} \label{955}
|h_r| \leq -C_0 \log |x| \qquad \hbox{in }B_{\lambda_1}(0)
\end{equation}
for some $C_0>0$;
\item or  $ \displaystyle \min_{|y|=1} H^{\lambda_n} \leq -\frac{n|\alpha+1|+1}{C}$ for a sequence $\lambda_n \downarrow 0$, which implies $\displaystyle \max_{|y|=1} H^{\lambda_n} \leq 0$ in view of \eqref{Harn} and in turn $h_r \leq 0$ on $|x|=\lambda_n$ for all $n \in \mathbb{N}$, leading to 
\begin{equation} \label{421}
h_r \leq 0\qquad \hbox{in }B_{\lambda_1}(0)
\end{equation}
by the weak maximum principle.
\end{itemize}
Notice that \eqref{421} implies the validity of \eqref{955} for some $C_0>0$ in view of Theorem 12-\cite{Ser1}. Thanks to \eqref{955} one can apply Theorem 1.1-\cite{KiVe} to show that
\begin{equation} \label{pro2}
h_r \in W^{1,q}(B) 
\end{equation}
for all $1\leq q <n$ and there exists $\gamma_r \in \mathbb{R}$ so that
\begin{equation} \label{sing1bis}
h_r-\gamma_r (n\omega_n |\gamma_r|^{n-2})^{-\frac{1}{n-1}} \log |x|\in L^\infty (B), \qquad \Delta_n h_r=\gamma_r \delta_0 \hbox{ in } B.
\end{equation}
In particular, $u \in \displaystyle \bigcap_{1\leq q<n} W^{1,q}(\Omega)$ in view of \eqref{pro1} and \eqref{pro2}, and \eqref{842} is established.
\end{proof} 

\medskip \noindent Even if $\gamma_r>-n^n|\alpha+1|^{n-2}(\alpha+1)\omega_n$, at this stage we cannot exclude that $\displaystyle \lim_{r \to 0}\gamma_r=-n^n|\alpha+1|^{n-2}(\alpha+1)\omega_n$. Therefore, we are not able to use \eqref{pro1} and \eqref{sing1bis} for improving the exponential integrability on $u$ to reach $|x|^{n\alpha} e^u=|x|^{n\alpha} e^{h_r} e^{u_r} \in L^p$ near $0$ for some $p>1$ and $r$ sufficiently small, as it would be necessary to prove $L^\infty-$bounds on $u_r$ via \eqref{907} on $u_{\epsilon,r}$. 

\medskip \noindent We need to argue in a different way. Since $u \in W^{1,n-1}(\Omega)$ in view of \eqref{842}, we can extend \eqref{E1} at $0$ as
\begin{equation} \label{1033}
-\Delta_n u= |x|^{n\alpha} e^u-\gamma \delta_0 \quad \hbox{in }\Omega.
\end{equation}
To see it, let $\varphi \in C_0^1(\Omega)$ and consider a function $\chi_\epsilon \in C^\infty(\Omega)$ with $0\leq \chi_\epsilon \leq 1$, $\chi_\epsilon=0$ in $B_{\frac{\epsilon}{2}}(0)$, $\chi_\epsilon=1$ in $\Omega \setminus B_\epsilon(0)$ and $\epsilon |\nabla \chi_\epsilon| \leq C$. Taking $\chi_\epsilon \varphi \in C_0^1(\Omega \setminus \{0\})$ as a test function in \eqref{E1} we have that
\begin{equation} \label{1022}
\int_\Omega |\nabla u|^{n-2}\langle \nabla u, \varphi \nabla \chi_\epsilon+\chi_\epsilon \nabla \varphi \rangle=\int_\Omega \chi_\epsilon |x|^{n\alpha} e^u  \varphi.
\end{equation}
Since $u \in W^{1,n-1}(\Omega)$ and $|x|^{n\alpha} e^u \in L^1(\Omega)$ it is easily seen that
\begin{equation} \label{1023}
\int_\Omega \chi_\epsilon  |\nabla u|^{n-2}\langle \nabla u, \nabla \varphi \rangle \to \int_\Omega |\nabla u|^{n-2}\langle \nabla u, \nabla \varphi \rangle ,\qquad \int_\Omega \chi_\epsilon |x|^{n\alpha} e^u  \varphi \to \int_\Omega |x|^{n\alpha} e^u \varphi 
\end{equation}
as $\epsilon \to 0$. Since
$$\int_\Omega |\nabla u|^{n-1}|\varphi - \varphi (0)| |\nabla \chi_\epsilon|
 \leq C \int_{B_\epsilon(0) \setminus B_{\frac{\epsilon}{2}}(0)}  |\nabla u|^{n-1} \to 0$$
as $\epsilon \to 0$ in view of $|\varphi - \varphi (0)|\leq C\epsilon$ in $B_\epsilon(0) \setminus B_{\frac{\epsilon}{2}}(0)$ and $u \in W^{1,n-1}(\Omega)$, the remaining term in \eqref{1022} can be re-written as follows:
\begin{equation} \label{1030}
\int_\Omega |\nabla u|^{n-2} \varphi  \langle \nabla u, \nabla \chi_\epsilon \rangle=
\varphi (0) \int_\Omega |\nabla u|^{n-2} \langle \nabla u,\nabla \chi_\epsilon \rangle+o(1)
\end{equation}
as $\epsilon \to 0$. By inserting \eqref{1023}-\eqref{1030} into \eqref{1022} we deduce the existence of
$$\gamma=\lim_{\epsilon \to 0} \int_\Omega |\nabla u|^{n-2} \langle \nabla u,\nabla \chi_\epsilon \rangle$$
and the validity of \eqref{1033} for $u$. Moreover, if we assume $u \in C^1(\overline{ \Omega} \setminus \{0\})$, we can interpret $\gamma$ as
$$\gamma=\lim_{\epsilon \to 0} [\int_\Omega |x|^{n\alpha} e^u \chi_\epsilon+ \int_{\partial \Omega} |\nabla u|^{n-2}\partial_n u]=\int_\Omega |x|^{n\alpha} e^u + \int_{\partial \Omega} |\nabla u|^{n-2}\partial_n u.$$

\medskip \noindent Since $\gamma_r \geq \gamma+o(1)$ as $r \to 0$ according to (4.16)-\cite{Esp}, we find that $h_r$ is possibly much lower than $u$ and then needs to be compensated by an unbounded function $u_r \geq 0$  in order to keep the validity of $u=u_r+h_r$. Instead, thanks to Theorem 2.1-\cite{KiVe} introduce $h \in \displaystyle \bigcap_{1\leq q<n} W^{1,q}(B)$ as the solution of 
$$\Delta_n h= \gamma \delta_0 \hbox{ in } B,\quad h=u \hbox{ on }\partial B$$
so that
\begin{equation} \label{954}
h-\gamma (n\omega_n |\gamma|^{n-2})^{-\frac{1}{n-1}} \log |x|\in L^\infty (B).
\end{equation}
Decomposing $u$ as $u=u_0+h$, the solution $u_0$ of \eqref{929} on $B$ is very likely a bounded function, as we will prove below.

\medskip \noindent In order to establish some crucial integral inequalities involving $u_0$, let us introduce the following approximation scheme. By convolution with mollifiers consider sequences  $f_j, g_j \in C_0^\infty(B)$ so that $f_j \rightharpoonup  |x|^{n \alpha} e^u-\gamma \delta_0$ weakly in the sense of measures and $0\leq f_j-g_j \to |x|^{n \alpha} e^u$ in $L^1(B)$ as $j \to +\infty$. Since $u \in C^{1,\eta}(\partial B)$, let $\varphi \in C^{1,\eta} (B)$ be the $n-$harmonic extension of $u \mid_{\partial B}$ in $B$. Let $v_j,w_j \in W_0^{1,n}(B)$ be the weak solutions of $-\hbox{div}\ {\bf a}(x,\nabla v_j)=f_j$ and $-\hbox{div}\ {\bf a}(x,\nabla w_j)=g_j$ in $B$, where ${\bf a}(x,p)=|p+\nabla \varphi|^{n-2} (p+\nabla \varphi)-|\nabla \varphi|^{n-2} \nabla \varphi$. In this way, $u_j=v_j+\varphi$ and $h_j=w_j+\varphi$ do solve
$$-\Delta_n u_j=f_j \hbox{ and } -\Delta_n h_j=g_j \hbox{ in }B,\: u_j=h_j=u \hbox{ on }\partial B.$$
Since $f_j,g_j$ are uniformly bounded in $L^1(B)$, by $(21)$ in \cite{BoGa} we can assume that $v_j \to v$ and $w_j \to w$ in $W^{1,q}_0(B)$ for all $1\leq q<n$ as $j \to +\infty$, where $v$ and $w$ do satisfy
\begin{equation} \label{636}
-\hbox{div}\ {\bf a}(x,\nabla v)=|x|^{n \alpha} e^u-\gamma \delta_0 \hbox{ and }-\hbox{div}\ {\bf a}(x,\nabla w)=-\gamma \delta_0 \hbox{ in }B
\end{equation}
in view of $g_j \rightharpoonup -\gamma \delta_0$ weakly in the sense of measures as $j \to +\infty$. Since $u-\varphi,h-\varphi \in \displaystyle \bigcap_{1\leq q<n} W^{1,q}_0(B)$ do solve the first and the second equation in \eqref{636}, respectively, by the uniqueness result in \cite{DHM} (see Theorems 1.2 and 4.2 in \cite{DHM}) we have that $v=u-\varphi$ and $w=h-\varphi$, i.e.
$$u_j \to u \hbox{ and } h_j \to h \hbox{ in } W^{1,q}(B) \hbox{ for all }1\leq q<n \hbox{ as }j \to +\infty.$$ Thanks to the approximation given by the $u_j$'s and $h_j$'s, we can now derive some crucial integral inequalities on $u_0$.
\begin{pro} \label{entropy} Let $u_0$ be a solution of \eqref{929}. Then $u_0 \geq 0$ and we have:
\begin{equation} \label{642}
\int_{\{k<|u_0| < k+a\} } |\nabla u_0|^n \leq \frac{a}{d} \int_B |x|^{n\alpha} e^u \qquad \forall \: k,a>0
\end{equation}
and, if $|x|^{n\alpha} e^u  \in L^p(B)$ for some $p>1$,
\begin{equation} \label{901}
\left(\int_B u_0^{2mnq} \right)^{\frac{1}{2m}} \leq \frac{C q^{n-1}}{d}   |B|^{\frac{n-1}{mnq}} \left( \int_B |x|^{n p \alpha}e^{pu} \right)^{\frac{1}{p}} \left(\int_B u_0^{mnq}\right)^{\frac{n(q-1)+1}{mnq}},
\end{equation}
where $m=\frac{p}{p-1}$ and 
\begin{equation} \label{745}
d= \inf_{X \not= Y} \frac{\langle |X|^{n-2}X-|Y|^{n-2}Y,X-Y \rangle}{|X-Y|^n}>0.
\end{equation}
\end{pro}
\begin{proof} First use $-(v_j-w_j)_- \in W^{1,n}_0(B)$ as a test function for $-\hbox{div}\ {\bf a}(x,\nabla v_j)+\hbox{div}\ {\bf a}(x,\nabla w_j)$ to get
$$ d \int_{\{ v_j-w_j<0\} } |\nabla (v_j-w_j)|^n  \leq   - \int_B \langle {\bf a}(x,\nabla v_j)-{\bf a}(x,\nabla w_j), \nabla (v_j-w_j)_- \rangle =-\int_B (f_j-g_j) (v_j-w_j)_- \leq 0$$
in view of \eqref{745} and $f_j-g_j\geq 0$. Hence, $v_j- w_j \geq 0$ and then $u_0\geq0 $ in view of $v_j-w_j \to u-h=u_0$ in $W^{1,q}_0(B)$ for all $1\leq q<n$ as $j \to +\infty$. Now, introduce the truncature operator $T_{k,a}$, for $k,a>0$, as
$$T_{k,a}(s)=\left\{ \begin{array}{cl}  s- k \hbox{ sign}(s) &\hbox{if }k< |s|< k+a, \\ a \hbox{ sign}(s)  &\hbox{if }|s|\geq  k+a,\\
 0 &\hbox{if }|s|\leq k,
\end{array} \right. $$
and use $T_{k,a}(v_j-w_j) \in W^{1,n}_0(B)$ as a test function for $-\hbox{div}\ {\bf a}(x,\nabla v_j)+\hbox{div}\ {\bf a}(x,\nabla w_j)$ to get
\begin{equation} \label{743}
 d \int_{\{ k<|v_j-w_j| < k+a\} } |\nabla (v_j-w_j)|^n  \leq   \int_B \langle {\bf a}(x,\nabla v_j)-{\bf a}(x,\nabla w_j), \nabla T_{k,a}(v_j-w_j) \rangle =\int_B (f_j-g_j) T_{k,a}(v_j-w_j)
\end{equation}
in view of \eqref{745}. Since $v_j-w_j \to u_0$ in $W^{1,q}_0(B)$ for all $1\leq q<n$ and $f_j-g_j \to |x|^{n \alpha} e^u$ in $L^1(B)$ as $j \to +\infty$, we can let $j \to +\infty$ in \eqref{743} and get by Fatou's Lemma that
$$ d \int_{\{ k<|u_0| < k+a\} } |\nabla u_0|^n  \leq   \int_B  |x|^{n \alpha} e^u T_{k,a}(u_0) \leq a  \int_B  |x|^{n \alpha} e^u$$ 
yielding the validity of \eqref{642}. Finally, if $|x|^{n\alpha} e^u  \in L^p(B)$ for some $p>1$, we can assume that $f_j-g_j \to |x|^{n \alpha} e^u$ in $L^p(B)$ as $j \to +\infty$ and use $T_a[|v_j-w_j|^{n(q-1)}(v_j-w_j)] \in W^{1,n}_0(B)$, where $T_a=T_{0,a}$ and $a>0,q \geq 1$, as a test function for $-\hbox{div}\ {\bf a}(x,\nabla v_j)+\hbox{div}\ {\bf a}(x,\nabla w_j)$ to get by H\"older's inequality
\begin{eqnarray*}
d \frac{n(q-1)+1}{q^n} \int_{\{ |v_j-w_j|^{n(q-1)+1} < a\} } |\nabla |v_j-w_j|^q|^n  \leq \int_B |f_j-g_j| |v_j-w_j|^{n(q-1)+1}\leq |B|^{\frac{n-1}{mnq}} \|f_j-g_j\|_p  \left(\int_B |v_j-w_j|^{mnq}\right)^{\frac{n(q-1)+1}{mnq}}
\end{eqnarray*}
in view of $|T_a(s)|\leq |s|$ and \eqref{745}. We have used that $v_j-w_j \in W^{1,n}_0(B) \subset  \displaystyle \bigcap_{q \geq 1} L^q(B)$ by the Sobolev embedding Theorem. Letting $a \to +\infty$, by Fatou's Lemma we get that
\begin{eqnarray*}
\int_B |\nabla |v_j-w_j|^q|^n \leq \frac{q^n}{d [n(q-1)+1]}   |B|^{\frac{n-1}{mnq}} \|f_j-g_j\|_p  \left(\int_B |v_j-w_j|^{mnq}\right)^{\frac{n(q-1)+1}{mnq}}.
\end{eqnarray*}
In particular, $|v_j-w_j|^q \in W^{1,n}_0(B)$ and by the Sobolev embedding  $W^{1,n}_0(B) \subset L^{2mn}(B)$ we have that
\begin{eqnarray*}
\left(\int_B |v_j-w_j|^{2mnq} \right)^{\frac{1}{2m}} \leq \frac{C q^n}{d [n(q-1)+1]}   |B|^{\frac{n-1}{mnq}} \|f_j-g_j\|_p  \left(\int_B |v_j-w_j|^{mnq}\right)^{\frac{n(q-1)+1}{mnq}}.
\end{eqnarray*}
Letting $j \to +\infty$, we finally deduce the validity of \eqref{901}:
\begin{eqnarray*}
\left(\int_B u_0^{2mnq} \right)^{\frac{1}{2m}} \leq \frac{C q^{n-1}}{d}   |B|^{\frac{n-1}{mnq}} \left( \int_B |x|^{n p \alpha}e^{pu} \right)^{\frac{1}{p}} \left(\int_B u_0^{mnq}\right)^{\frac{n(q-1)+1}{mnq}}
\end{eqnarray*}
in view of $v_j-w_j \to u_0$ in $W^{1,q}_0(B)$ for all $1\leq q<n$ and $f_j-g_j \to |x|^{n \alpha} e^u$ in $L^p(B)$ as $j \to +\infty$.
\end{proof}

\medskip \noindent We are now ready to complete the proof of Theorem \ref{thm1}.\\

\medskip \noindent \emph{Proof (of Theorem \ref{thm1}).} Since $u_0 \geq 0$ by Proposition \ref{entropy}, we have that $h \leq u$. By \eqref{E1} and \eqref{954} we have that
$$\int_{B}  |x|^{n\alpha+\gamma (n\omega_n |\gamma|^{n-2})^{-\frac{1}{n-1}}} \leq C 
\int_{B}  |x|^{n\alpha}e^h  \leq C \int_\Omega |x|^{n\alpha}e^u<+\infty,$$
which implies 
\begin{equation} \label{459}
n\alpha+\gamma (n\omega_n |\gamma|^{n-2})^{-\frac{1}{n-1}}>-n
\end{equation}
or equivalently
\begin{equation} \label{504}
\gamma>-n^n|\alpha+1|^{n-2}(\alpha+1)\omega_n.
\end{equation}
Since $|x|^{n\alpha}e^h \in L^1$ near $0$ and $h$ has a logarithmic singularity at $0$, then, as already observed in the Introduction, a stronger integrability follows:
\begin{equation} \label{505}
|x|^{n\alpha} e^h \in L^p(B) 
\end{equation}
for some $p>1$. Inequality \eqref{642} is used in \cite{AgPe} to deduce exponential estimates on $u_0$ like
\begin{equation} \label{1011}
\int_B e^{\frac{\delta u_0}{\|f\|_1}} \leq C_r
\end{equation}
for some $\delta>0$ where $f=|x|^{n\alpha} e^u$. Since $\displaystyle \lim_{r \to 0} \int_{B} |x|^{n \alpha} e^u=0$, by \eqref{1011} we deduce that $e^{u_0} \in L^p(B)$ for all $p\geq 1$ if $r$ is sufficiently small and then by \eqref{505}
$$|x|^{n\alpha} e^u=|x|^{n\alpha} e^h e^{u_0} \in L^p(B) $$
for some $p>1$. Inequality \eqref{901} is used in Proposition 4.1-\cite{Esp} (compare with (4.4) in \cite{Esp}) to get $u_0 \in L^\infty(B)$ and then \eqref{954} does hold for $u$ too, yielding the validity of \eqref{sing1}. In order to prove \eqref{gamma1}, set $H=u-\gamma  (n\omega_n |\gamma|^{n-2})^{-\frac{1}{n-1}} \log |x|$ and introduce the function
$$U_r(y)=u(ry)-\gamma  (n\omega_n |\gamma|^{n-2})^{-\frac{1}{n-1}} \log r=
\gamma  (n\omega_n |\gamma|^{n-2})^{-\frac{1}{n-1}} \log |y|+H(ry)$$
for a given sequence $r \to 0$. Since
$$-\Delta_n U_r= r^{n(1+\alpha)} |y|^{n\alpha}  e^{u(ry)}=
r^{n(1+\alpha) +\gamma  (n\omega_n |\gamma|^{n-2})^{-\frac{1}{n-1}}}  |y|^{n\alpha+\gamma  (n\omega_n |\gamma|^{n-2})^{-\frac{1}{n-1}}} e^{H(ry)},$$
by \eqref{sing1} and \eqref{459}-\eqref{504} we have that $U_r$ and $\Delta_n U_r$ are bounded in $L^\infty_{\hbox{loc}}(\mathbb{R}^n \setminus \{0\})$, uniformly in $r$. By \cite{Dib,Ser1,Tol} we deduce that $U_r$ is bounded in $C^{1,\eta}_{\hbox{loc}}(\mathbb{R}^n \setminus \{0\})$, uniformly in $r$. By the Ascoli-Arzel\'a's Theorem and a diagonal process, up to a sub-sequence we have that $U_r \to U_0$ in $C^1_{\hbox{loc}}(\mathbb{R}^n \setminus \{0\})$, where $U_0$ is a n-harmonic function in $\mathbb{R}^n \setminus \{0\}$. Setting $H_r(y)=H(ry)$, we deduce that $H_r \to H_0$ in $C^1_{\hbox{loc}}(\mathbb{R}^n \setminus \{0\})$, where $H_0 \in L^\infty(\mathbb{R}^n)$ in view of \eqref{sing1}.
Since $U_0=\gamma  (n\omega_n |\gamma|^{n-2})^{-\frac{1}{n-1}} \log |y|+H_0$ with $H_0 \in L^\infty (\mathbb{R}^n)\cap C^1(\mathbb{R}^n \setminus \{0\})$, it is well known that $H_0$ is a constant function, as shown in Corollary 2.2-\cite{KiVe} (see also \cite{Esp} for a direct proof). In particular we get that
$$\sup_{|x|=r} |x|  \Big| \nabla [ u-\gamma  (n\omega_n |\gamma|^{n-2})^{-\frac{1}{n-1}} \log |x|]\Big| =\sup_{|y|=1} |\nabla H_r| \to \sup_{|y|=1} |\nabla H_0| =0.$$
Since this is true for any sequence $r \to 0$ up to extracting a sub-sequence, we have established the validity of \eqref{gamma1}. The proof of Theorem \ref{thm1} is concluded.
\begin{flushright}
$\Box$
\end{flushright}

\section{Quantization results}
\noindent In this section we will make crucial use of the following integral identity: for any solution $u$ of \begin{equation}\label{338}
-\Delta_n u=|x|^{n\alpha} e^u \hbox{ in } \mathbb{R}^n \setminus \{0\}
\end{equation}
there holds
\begin{equation}\label{339}
n (\alpha+1) \int_A |x|^{n\alpha} e^u=\int_{\partial A} \left[|x|^{n\alpha} e^u \langle x,\nu\rangle 
+|\nabla u|^{n-2} \partial_{\nu}u \ \langle \nabla u,x \rangle -\frac{|\nabla u|^n}{n}\langle x,\nu\rangle  \right],
\end{equation}
where $A$ is the annulus $A=B_R(0) \setminus B_\epsilon(0)$, $0<\epsilon<R<+\infty$, and $\nu$ is the unit outward normal vector at $\partial A$. Notice that \eqref{339} is simply a special case of the well-known Pohozaev identities associated to \eqref{338}. Even though the classical Pohozaev identities require more regularity than simply $u \in C^{1,\eta}(\mathbb{R}^n \setminus \{0\})$, \eqref{339} is still valid  in the quasilinear case and we refer to \cite{DFSV} for a justification. Thanks to \eqref{339} we are able to show the following general result.
\begin{pro} \label{Poho}
Let $u$ be a solution of \eqref{338} so that \eqref{sing1}-\eqref{gamma1} do hold at $0$ and $\infty$ with $\gamma$ and $-\gamma_\infty$, respectively,  so that  $\gamma> - n^n |\alpha+1|^{n-2}(\alpha+1) \omega_n$ and $\gamma_\infty> n^n |\alpha+1|^{n-2}(\alpha+1) \omega_n$. Then $\int_{\mathbb{R}^n} |x|^{n \alpha} e^u =\gamma +\gamma_\infty$ satisfies
\begin{equation} \label{615}
n (\alpha+1) (\gamma +\gamma_\infty)=
\frac{n-1}{n} (n\omega_n)^{-\frac{1}{n-1}}   \left[|\gamma_\infty|^{\frac{n}{n-1}} 
-|\gamma|^{\frac{n}{n-1}} \right].
\end{equation}
\end{pro}
\begin{proof}
By \eqref{sing1}-\eqref{gamma1} at $0$ with $\gamma> - n^n |\alpha+1|^{n-2}(\alpha+1) \omega_n $ we deduce that
\begin{equation} \label{1931}
 |\nabla u|=\frac{1}{|x|} \left[(\frac{|\gamma|}{n\omega_n})^{\frac{1}{n-1}} + o(1)\right],\quad
\langle \nabla u, x \rangle=\gamma  (n\omega_n |\gamma|^{n-2})^{-\frac{1}{n-1}}+o(1),\quad
|x|^{n\alpha}e^u=o(\frac{1}{|x|^n})
\end{equation}
as $x \to 0$ thanks to the equivalence between \eqref{459} and \eqref{504}. By \eqref{1931} we have that
\begin{equation} \label{2115}
\int_{\partial B_\epsilon(0)} |x| \left[|x|^{n\alpha} e^u 
+|\nabla u|^{n-2} \langle \nabla u, \frac{x}{|x|} \rangle^2-\frac{|\nabla u|^n}{n}  \right]\to 
 \frac{n-1}{n} (n\omega_n)^{-\frac{1}{n-1}} |\gamma|^{\frac{n}{n-1}} 
\end{equation}
as $\epsilon \to 0^+$ in view of $\hbox{Area}(\mathbb{S}^{n-1})= n\omega_n$. Similarly, by \eqref{sing1}-\eqref{gamma1} at $\infty$ with $-\gamma_\infty$ so that $\gamma_\infty> n^n |\alpha+1|^{n-2}(\alpha+1) \omega_n$ we deduce that
$$ |\nabla u|=\frac{1}{|x|} \left[(\frac{|\gamma_\infty|}{n\omega_n})^{\frac{1}{n-1}} + o(1)\right],\quad
\langle \nabla u, x \rangle=-\gamma_\infty  (n\omega_n |\gamma_\infty|^{n-2})^{-\frac{1}{n-1}}+o(1),\quad
|x|^{n\alpha}e^u=o(\frac{1}{|x|^n})$$
as $|x| \to \infty$ and then
\begin{equation} \label{522}
\int_{\partial B_R(0)} |x| \left[|x|^{n\alpha} e^u 
+|\nabla u|^{n-2} \langle \nabla u,\frac{x}{|x|} \rangle^2-\frac{|\nabla u|^n}{n}  \right]\to 
\frac{n-1}{n}(n \omega_n)^{-\frac{1}{n-1}}   |\gamma_\infty|^{\frac{n}{n-1}} 
\end{equation}
as $R \to +\infty$. In view of \eqref{gamma1} at $0$ and $\infty$ we easily get that
$$-\Delta_n u=|x|^{n \alpha} e^u -\gamma \delta_0 -\gamma_\infty \delta_{\infty} \hbox{ in } \mathbb{R}^n$$
in the sense
$$\int_{\mathbb{R}^n} |\nabla u|^{n-2} \langle \nabla u,\nabla \varphi \rangle=\int_{\mathbb{R}^n} |x|^{n \alpha} e^u \varphi -\gamma \varphi(0) -\gamma_\infty \varphi(\infty)$$
for all $\varphi \in C^1(\mathbb{R}^n)$ so that $\varphi(\infty):=\displaystyle \lim_{|x| \to \infty} \varphi(x)$ does exist. Choosing $\varphi=1$ we deduce that
\begin{equation} \label{556}
\int_{\mathbb{R}^n} |x|^{n \alpha} e^u =\gamma +\gamma_\infty.
\end{equation}
By inserting \eqref{2115}-\eqref{556} into \eqref{339} and letting $\epsilon \to 0^+,\: R\to +\infty$ we deduce the validity of \eqref{615}. 

\end{proof}
\noindent Let us now apply Proposition \ref{Poho} to problems \eqref{E1bisentire} and \eqref{E1sim}.

\medskip \noindent \emph{Proof (of Theorem \ref{thm2}).} Let $u$ be a solution of \eqref{E1bisentire}. By Theorem \ref{thm1} \eqref{sing1}-\eqref{gamma1} do hold for $u$ at $0$ with $\gamma> - n^n \omega_n$. By \eqref{Kelvin} the Kelvin transform $\hat u$ satisfies
$$-\Delta_n \hat u=|x|^{-2n} e^{\hat u} \hbox{ in } \mathbb{R}^n \setminus \{0\}.$$
Let us apply Theorem \ref{thm1} to deduce the validity of \eqref{sing1}-\eqref{gamma1} for $\hat u$ at $0$ with $\gamma_\infty>  n^n \omega_n $. Back to $u$, \eqref{sing1}-\eqref{gamma1} do hold for $u$ at $\infty$ with $-\gamma_\infty$ so that $\gamma_\infty> n^n \omega_n$. Let us apply Proposition \ref{Poho} with $\alpha=0$ to get $\int_{\mathbb{R}^n} e^u =\gamma +\gamma_\infty$ with $\gamma_\infty$ satisfying \eqref{921}.
Notice that the function $f(s)=n s+\frac{n-1}{n} (n\omega_n)^{-\frac{1}{n-1}}  |s|^{\frac{n}{n-1}}$ is increasing in $(-n^n \omega_n,+\infty)$ and then $f(s) >f(-n^n \omega_n)=-n^n \omega_n$ for all $s \in (-n^n \omega_n,+\infty)$. At the same time the function $g(s)=\frac{n-1}{n} (n\omega_n)^{-\frac{1}{n-1}}   s^{\frac{n}{n-1}} -n  s$ is increasing in $(n^n \omega_n,+\infty)$ and then $g(s)>g(n^n \omega_n)=-n^n \omega_n$ for all $s \in (n^n\omega_n,+\infty)$. Therefore, for any $\gamma>- n^n \omega_n$ equation \eqref{921} has a unique solution $\gamma_\infty>n^n \omega_n$. The proof of Theorem \ref{thm2} is concluded.
\begin{flushright}
$\Box$
\end{flushright}
\begin{oss}
Concerning Corollary \ref{cor1}, observe that in the argument above we have established \eqref{1455} for problem \eqref{E1entire} on $\Omega=\mathbb{R}^n$ and a similar proof is in order for a general unbounded open set $\Omega$. Since $\gamma=0$, we deduce the validity of \eqref{145} in view of \eqref{quant}.
\end{oss}

\medskip \noindent \emph{Proof (of Theorem \ref{thm3}).}  Let $v$ be a solution of \eqref{E1sim}. Applying Theorem \ref{thm1} to the Kelvin transform $\hat v$, solution of
$$-\Delta_n \hat v=|x|^{-n(\alpha+2)} e^{\hat v} \hbox{ in } \mathbb{R}^n \setminus \{0\},$$
we deduce the validity of \eqref{sing1}-\eqref{gamma1} for $v$ at $\infty$ with $-\gamma_\infty$ so that
$\gamma_\infty> n^n |\alpha+1|^{n-2}(\alpha+1) \omega_n$. By Proposition \ref{Poho} with $\gamma=0$ we deduce that
$\gamma_\infty=\int_{\mathbb{R}^n} |x|^{n\alpha} e^v$ satisfies
$$n (\alpha+1) \gamma_\infty=
\frac{n-1}{n} (n\omega_n)^{-\frac{1}{n-1}}   \gamma_\infty^{\frac{n}{n-1}} .$$
Therefore, $\alpha> -1$ and
$$\int_{\mathbb{R}^n} |x|^{n \alpha} e^v=n(\frac{n^2}{n-1})^{n-1} (\alpha+1)^{n-1} \omega_n  ,$$
concluding the proof of Theorem \ref{thm3}.
\begin{flushright}
$\Box$
\end{flushright}

\section{Radial solutions for \eqref{E1bisentire}}
\noindent Fix $M>1$ and assume that
\begin{equation} \label{uniform}
\frac{1}{M}\leq r_0  \leq M, \quad \alpha_0 \leq M, \quad \frac{1}{M}\leq |\alpha_1| \leq M.
\end{equation}
Let us first discuss the local existence theory for the following Cauchy problem: 
\begin{equation} \label{Cauchy}
\left\{ \begin{array}{l}
-\frac{1}{r^{n-1}}(r^{n-1} |U'|^{n-2} U')'=e^U \\
\: U(r_0)=\alpha_0,\:\:\: U'(r_0)=\alpha_1. \end{array} \right.
\end{equation}
Given $0<\delta<\frac{1}{2M}$, define $I=[r_0-\delta,r_0+\delta]$ and $E=\{U \in C(I,[\alpha_0-1,\alpha_0+1]):\ U(r_0)=\alpha_0 \}$, which is a Banach space endowed with $\|\cdot\|_\infty$ as a norm. We can re-formulate \eqref{Cauchy} as $U=TU$, where
$$TU(r)=\alpha_0+\int_{r_0}^r \frac{ds}{s} \Big| r_0^{n-1}|\alpha_1|^{n-2}\alpha_1  -\int_{r_0}^s t^{n-1} e^{U(t)} dt \Big|^{-\frac{n-2}{n-1}} \left(r_0^{n-1} |\alpha_1|^{n-2}\alpha_1-\int_{r_0}^s t^{n-1} e^{U(t)} dt\right).$$
In view of
\begin{equation}\label{858}
|s^n-r_0^n| \leq n (M+1)^{n-1} \delta \qquad \forall\: s \in I
\end{equation}
we have that $\displaystyle \max_I U \leq M+1$ and $\displaystyle \max_I |\int_{r_0}^s t^{n-1} e^{U(t)} dt| \leq e^{M+1} (M+1)^{n-1} \delta$ for all $U \in E$, and then for $0<\delta <  \frac{e^{-M-1}}{2 (M+1)^{3n-3}}$ we have that
\begin{equation} \label{938}
r_0^{n-1} |\alpha_1|^{n-2}\alpha_1-\int_{r_0}^s t^{n-1} e^{U(t)} dt \hbox{ has the same sign as } \alpha_1 \:\: \forall s \in I
\end{equation}
and
\begin{equation} \label{939}
\frac{1}{2 M^{2n-2}} \leq \frac{1}{2} r_0^{n-1} |\alpha_1|^{n-1} \leq |r_0^{n-1} |\alpha_1|^{n-2}\alpha_1-\int_{r_0}^s t^{n-1} e^{U(t)} dt| \leq \frac{3}{2} r_0^{n-1} |\alpha_1|^{n-1} 
\leq \frac{3}{2} M^{2n-2}
\end{equation}
for all $U \in E$. Since $\log \frac{r_0+\delta}{r_0} \leq \log \frac{r_0}{r_0-\delta} \leq \frac{\delta}{r_0-\delta} \leq 2 M \delta$ in view of $\delta<\frac{r_0}{2}$ and 
$$||x|^{-\frac{n-2}{n-1}}x-|y|^{-\frac{n-2}{n-1}}y| \leq C_M |x-y| \quad \forall \: x,y \in \mathbb{R}: \: xy\geq 0, \: \min \{|x|,|y|\}\geq \frac{1}{2 M^{2n-2}}$$
(for example, take $C_M=(1+\frac{n-2}{n-1})(4M^{2n-2})^{\frac{n-2}{n-1}}$), by \eqref{938}-\eqref{939} we have that
$$\|TU-\alpha_0 \|_{\infty,I} \leq  \sup_{r \in I} |\int_{r_0}^r \frac{ds}{s} \Big| r_0^{n-1}|\alpha_1|^{n-2}\alpha_1  -\int_{r_0}^s t^{n-1} e^{U(t)} dt \Big|^{\frac{1}{n-1}}| \leq 2 (\frac{3}{2})^{\frac{1}{n-1}} M^3 \delta \leq 3 M^3 \delta$$
and
$$\|TU-TV\|_{\infty,I}\leq C_M \sup_{r \in I} |\int_{r_0}^r \frac{ds}{s} |\int_{r_0}^s t^{n-1} [e^{U(t)}-e^{V(t)}] dt| |\leq 2 C_M (M+1)^n e^{M+1} \delta \|U-V\|_{\infty,I} $$
for all $U,V \in C^1 (I)$ in view of $\delta <1$ and \eqref{858}. In conclusion, if 
\begin{equation} \label{1251}
0<\delta<\min \{\frac{1}{3M^3}, \frac{e^{-M-1}}{2 (M+1)^{3n-3}},  \frac{e^{-M-1}}{2C_M (M+1)^n}  \},
\end{equation}
then $T$ is a contraction map from $E$ into itself and a unique fixed point $U \in E$ of $T$ is found by the Contraction Mapping Theorem, providing a solution $U$ of \eqref{Cauchy} in $I=[r_0-\delta,r_0+\delta]$.

\medskip \noindent Once a local existence result has been established for \eqref{Cauchy}, we can turn the attention to global issues. Given $r_0>0$, $\alpha_0$ and $\alpha_1 \not=0$, let $I=(r_1,r_2)$, $0\leq r_1<r_0<r_2\leq +\infty $, be the maximal interval of existence for the solution $U$ of \eqref{Cauchy}. We claim that $r_1=0$ when $\alpha_1>0$ and $r_2=+\infty$ when $\alpha_1<0$.

\medskip \noindent Consider first the case $\alpha_1>0$ and assume by contradiction $r_1>0$. Since 
\begin{equation} \label{1206}
U'(r) =\frac{1}{r} \left(r_0^{n-1}\alpha_1^{n-1}+\int_r^{r_0} t^{n-1} e^{U(t)} dt \right)^{\frac{1}{n-1}} \geq \frac{r_0 \alpha_1}{r}>0
\end{equation}
for all $r \in (r_1,r_0]$, one would have that 
$$U(r) \leq \alpha_0,\quad \alpha_1 \leq U'(r) \leq 
\frac{1}{r_1} [r_0^{n-1} \alpha_1^{n-1}+\frac{r_0^n}{n} e^{\alpha_0}]^{\frac{1}{n-1}}$$
for all $r\in (r_1,r_0]$ and then \eqref{uniform} would hold for initial conditions $\alpha_0'=U(r_0')$, $\alpha_1'=U'(r_0')$ in \eqref{Cauchy} at $r_0'$ approaching $r_1$ from the right. Since this would allow to continue the solution $U$ on the left of $r_1$ in view of the estimate \eqref{1251} on the time for local existence, we would reach a contradiction and then the property $r_1=0$ has been established.

\medskip \noindent In the case $\alpha_1<0$ assume by contradiction $r_2<+\infty$. Since
\begin{equation} \label{1207}
U'(r) =- \frac{1}{r} \left(r_0^{n-1}|\alpha_1|^{n-1}+\int_{r_0}^r t^{n-1} e^{U(t)} dt \right)^{\frac{1}{n-1}} \leq - \frac{r_0 |\alpha_1|}{r}<0
\end{equation}
for all $r \in [r_0,r_2)$, one would have that 
$$U(r) \leq \alpha_0,\quad -\frac{1}{r_0} [r_0^{n-1} |\alpha_1|^{n-1}+\frac{r_2^n}{n} e^{\alpha_0}]^{\frac{1}{n-1}} \leq U'(r) \leq - \frac{r_0 |\alpha_1|}{r_2} 
$$
for all $r\in [r_0,r_2)$ and then \eqref{uniform} would hold for initial conditions $\alpha_0'=U(r_0')$, $\alpha_1'=U'(r_0')$ in \eqref{Cauchy} at $r_0'$ approaching $r_2$ from the left. Since one could continue the solution $U$ past $r_2$ thanks to \eqref{1251}, a contradiction would arise. Then, we have shown that $r_2=+\infty$.

\medskip \noindent Given $\epsilon>0$, let now $U_\epsilon^\pm$ be the maximal solution of
$$\left\{ \begin{array}{l}
-\frac{1}{r^{n-1}}(r^{n-1} |U'|^{n-2} U')'=e^U \\
\: U(1)=\alpha_0,\:\:\: U'(1)=\pm \epsilon.\end{array} \right. $$
By the discussion above we have that $U_\epsilon^+$ and $U_\epsilon^-$ are well defined in $(0,1]$ and $[1,+\infty)$, respectively. According to \eqref{1206}-\eqref{1207} one has
\begin{equation} \label{405}
(U_\epsilon^+)' =\frac{1}{r} \left(\epsilon^{n-1}+\int_r^1 t^{n-1} e^{U_\epsilon^+(t)} dt \right)^{\frac{1}{n-1}} \hbox{ in }(0,1],\quad (U_\epsilon^-)'=- \frac{1}{r} \left(\epsilon^{n-1}+\int_1^r t^{n-1} e^{U_\epsilon^-(t)} dt \right)^{\frac{1}{n-1}}  \hbox{ in }[1,+\infty)
\end{equation}
and then $U_\epsilon^+$, $U_\epsilon^-$ are uniformly bounded in $C^{1,\gamma}_{loc} (0,1]$, $C^{1,\gamma}_{loc} [1,+\infty)$, respectively, in view of $U_\epsilon^+,U_\epsilon^- \leq \alpha_0$. Up to a subsequence and a diagonal argument, we can assume that $U_\epsilon^+ \to U^+$ in $C^1_{loc} (0,1]$ and $U_\epsilon^- \to U^-$ in $C^1_{loc} [1,+\infty)$ as $\epsilon \to 0^+$, where 
\begin{equation} \label{1208}
(U^+)'=\frac{1}{r} \left(\int_r^1 t^{n-1} e^{U^+(t)} dt \right)^{\frac{1}{n-1}} \hbox{ in }(0,1],\quad
(U^-)'=- \frac{1}{r} \left(\int_1^r t^{n-1} e^{U_-(t)} dt \right)^{\frac{1}{n-1}}  \hbox{ in }[1,+\infty)
\end{equation}
thanks to \eqref{405}. Since  $U^+(1)=U^-(1)=\alpha_0$ and $(U^+)'(1)=(U^-)'(1)=0$ in view of \eqref{1208}, we have that
$$U=\left\{ \begin{array}{ll} U^+ & \hbox{in }(0,1]\\ U^- & \hbox{in }[1,+\infty)\end{array}\right.$$ 
is in $C^1(0,+\infty)$ with $U\leq U(1)=\alpha_0$, $U'(1)=0$ and
\begin{equation} \label{433}
U'(r) =\frac{1}{r} \Big| \int_r^1 t^{n-1} e^{U(t)} dt \Big|^{-\frac{n-2}{n-1}} \int_r^1 t^{n-1} e^{U(t)} dt  \hbox{ in }(0,+\infty).
\end{equation}
It is not difficult to check that $U$ satisfies $-\Delta_n U=e^U$ in $\mathbb{R}^n \setminus \{0 \}$ and
\begin{equation} \label{singularity}
\lim_{r \to 0} \frac{U(r)}{\log r}=\lim_{ r \to 0} rU'(r)=( \int_0^1 t^{n-1} e^{U(t)} dt )^{\frac{1}{n-1}}=
\left( \frac{1}{n \omega_n} \int_{B_1(0)} e^U  \right)^{\frac{1}{n-1}}
\end{equation}
in view of \eqref{433}. By Theorem \ref{thm1} and \eqref{singularity}  we deduce that $U$ is a radial solution of
$$-\Delta_n U=e^U -\gamma \delta_0 \hbox{ in } \mathbb{R}^n ,\quad U \leq U(1)=\alpha_0,$$
with $\gamma=\int_{B_1(0)} e^{U} $ depending on the choice of $\alpha_0$. By the Pohozaev identity \eqref{339} on $A=B_1(0) \setminus B_\epsilon(0)$, $\epsilon \in (0,1)$, we have that
$$\omega_n[e^{\alpha_0}- \epsilon^n e^{U(\epsilon)}]=\int_{B_1(0)\setminus B_\epsilon(0)} e^U +\frac{n-1}{n} \omega_n [\epsilon U'(\epsilon)]^n  $$
in view of $U(1)=\alpha_0$ and $U'(1)=0$, and letting $\epsilon \to 0^+$ one deduces that
$$\omega_n e^{\alpha_0}= \gamma+\frac{n-1}{n} \omega_n \left( \frac{\gamma}{n \omega_n}  \right)^{\frac{n}{n-1}} $$
in view of \eqref{singularity}. Since $\gamma \in (0,+\infty) \to \gamma+\frac{n-1}{n} \omega_n \left( \frac{\gamma}{n \omega_n}  \right)^{\frac{n}{n-1}} \in (0,+\infty)$ is a bijection, for any given $\gamma>0$ let $\alpha_0=  \log[\frac{\gamma}{\omega_n}+\frac{n-1}{n} ( \frac{\gamma}{n \omega_n})^{\frac{n}{n-1}}]$ and the corresponding $U$ is the solution of \eqref{E1bisentire} we were searching for. Notice that $\int_{\mathbb{R}^n} e^U<+\infty$ in view of $\int_1^\infty t^{n-1} e^{U(t)} dt<+\infty$, as it can be deduced by
$$\lim_{r \to +\infty} \frac{U(r)}{\log r}=\lim_{ r \to +\infty} rU'(r)=-\left( \int_1^\infty t^{n-1} e^{U(t)} dt \right)^{\frac{1}{n-1}}$$
due to \eqref{433}. We have established the following result:
\begin{thm} \label{radialsol} For any $\gamma>0$ there exists a $1-$parameter family of distinct solutions $U_\lambda$, $\lambda>0$, to \eqref{E1bisentire} given by $U_\lambda(x)=U(\lambda x)+n \log \lambda$ such that $U_\lambda$ takes its unique absolute maximum point at $\frac{1}{\lambda}$.
\end{thm}


\bibliographystyle{plain}

\begin{thebibliography}{99}

\bibitem{AgPe} J.A. Aguilar Crespo, I. Peral Alonso, \emph{Blow-up behavior for solutions of $-\Delta_{N} u=V(x)e^{u}$ in bounded domains in $\mathbb{R}^{N}$}. Nonlinear Anal. {\bf 29} (1997), no. 4, 365--384.


\bibitem{BaTa} D. Bartolucci, G. Tarantello, \emph{Liouville type equations with singular data and their applications to periodic multivortices for the electroweak theory}. Comm. Math. Phys. {\bf 229} (2002), no. 1, 3--47. 

\bibitem{BBGGPV} P. B\'{e}nilan, L. Boccardo, T. Gallou\"{e}t, R. Gariepy, M. Pierre, J.L. V\'azquez, \emph{An $L^{1}$-theory of existence and uniqueness of solutions of nonlinear elliptic equations}. Ann. Scuola Norm. Sup. Pisa Cl. Sci. (4) {\bf 22} (1995), no. 2, 241--273.

\bibitem{BoGa} L. Boccardo, T. Gallou\"{e}t, \emph{Nonlinear elliptic and parabolic equations involving measure data}. J. Funct. Anal. {\bf 87} (1989), no. 1, 149--169.


\bibitem{BrMe} H. Br\'{e}zis, F. Merle, \emph{Uniform estimates and blow-up behavior for solutions of $-\Delta u=V(x)e^{u}$ in two dimensions}. Comm. Partial Differential Equations {\bf 16} (1991), no. 8-9, 1223--1253.



\bibitem{ChLi} W. Chen, C. Li, \emph{Classification of solutions of some nonlinear elliptic equations}. Duke Math. J. {\bf 63} (1991), no. 3, 615--623.



\bibitem{ChWa} K.S. Chou, T.Y.H. Wan, \emph{Asymptotic radial symmetry for solutions of $\Delta u+e^u=0$ in a punctured disc}. Pacific J. Math. {\bf 163} (1994), 269--276.

\bibitem{DFSV} L. Damascelli, A. Farina, B. Sciunzi, E. Valdinoci, \emph{Liouville results for m-Laplace equations of Lame-Emden-Fowler type}. Ann. Inst. H. Poincar\'e Anal. Non Lin\'{e}aire {\bf 26} (2009), no. 4, 1099--1119.




\bibitem{Dib} E. DiBenedetto, \emph{$C^{1+\alpha}$ local regularity of weak solutions of degenerate elliptic equations}. Nonlinear Anal. {\bf 7} (1983), no. 8, 827--850.

\bibitem{DHM} G. Dolzmann, N. Hungerb\"uhler, S. M\"uller, \emph{Uniqueness and maximal regularity for nonlinear elliptic systems of n-Laplace type with measure valued right hand side}.
J. Reine Angew. Math. {\bf 520} (2000), 1--35. 

\bibitem{Esp} P. Esposito, \emph{A classification result for the quasi-linear Liouville equation}. Ann. Inst. H. Poincar\'e Anal. Non Lin\'{e}aire  35 (2018), no. 3, 781--801.

\bibitem{EsMo} P. Esposito, F. Morlando, \emph{On a quasilinear mean field equation with an exponential nonlinearity}. J. Math. Pures Appl. (9) {\bf 104} (2015), no. 2, 354--382.




\bibitem{KiVe} S. Kichenassamy, L. Veron, \emph{Singular solutions of the $p$-Laplace equation}. Math. Ann. {\bf 275} (1986), no. 4, 599--615.




\bibitem{LiSh} Y.Y. Li, I. Shafrir, \emph{Blow-up analysis for solutions of $-\Delta u=Ve^{u}$ in dimension two}. Indiana Univ. Math. J. {\bf 43} (1994), no. 4, 1255--1270.

\bibitem{Lie} G.M. Lieberman, \emph{Boundary regularity for solutions of degenerate elliptic equations}. Nonlinear Anal. {\bf 12} (1988), no. 11, 1203--1219.


\bibitem{Lio} J. Liouville, \emph{ Sur l'{\'e}quation aud d{\'e}riv{\'e}es partielles $\partial^2 \:\mbox{log }\lambda/\partial u \partial v\pm 2 \lambda a^2=0$}, J. de Math. {\bf 18} (1853), 71--72. 



\bibitem{PrTa} J. Prajapat, G. Tarantello, \emph{On a class of elliptic problems in $\mathbb{R}^2$: symmetry and uniqueness results}, Proc. Royal Soc. Edinburgh. {\bf 131} (2001), no. 4, 967--985.


\bibitem{RoWe} F. Robert, J. Wei, \emph{Asymptotic behavior of a fourth order mean field equation with Dirichlet boundary condition}. Indiana Univ. Math. J. {\bf 57} (2008), no. 5, 2039--2060. 



\bibitem{Ser1} J. Serrin, \emph{Local behavior of solutions of quasilinear equations}. Acta Math. {\bf 111} (1964), 247--302.

\bibitem{Ser2} J. Serrin, \emph{Isolated singularities of solutions of quasi-linear equations}. Acta Math. {\bf 113 } (1965), 219--240. 


\bibitem{Tar} G. Tarantello, \emph{A quantization property for blow-up solutions of singular Liouville-type equations}. J. Funct. Anal. 219 (2005), 368--399.


\bibitem{Tol} P. Tolksdorf, \emph{Regularity for  more general class of quasilinear elliptic equations}. J. Differential Equations {\bf 51} (1984), no. 1, 126--150.


 



\end{thebibliography}

\end{document}